\documentclass[a4paper,12pt]{amsart}

\usepackage{amsmath,amsthm,amsfonts,amssymb,stmaryrd,mathrsfs,enumitem}
\usepackage{latexsym}
\usepackage{fullpage}
\usepackage{a4,color,palatino,fancyhdr}
\usepackage{graphicx}
\usepackage{float}  
\usepackage[small, bf, margin=90pt, tableposition=bottom]{caption}
\RequirePackage{amssymb}
\RequirePackage[T1]{fontenc}
\usepackage{amscd}
\usepackage{xypic}
\input{xy}
\xyoption{all}
\xyoption{poly}
\usepackage[all]{xy}
\usepackage{tikz}
\newtheorem{theorem}{Theorem}[section]
\newtheorem{proposition}[theorem]{Proposition}
\newtheorem{definition}[theorem]{Definition}

\newtheorem{corollary}[theorem]{Corollary}
\newtheorem{problem}{Problem}
\newtheorem{remark}{Remark}

\newtheorem{conj}[theorem]{Conjecture}

\numberwithin{equation}{section}

\def\qed{\hfill \mbox{$\square$}}

\begin{document}

\title{Rigidity of powers and Kosniowski's conjecture}
\author {Zhi L\"u and Oleg R. Musin}

 \subjclass[2010]{}
\thanks{The first author is partially supported by the  NSFC  grants (No. 11661131004, 11371093 and 11431009). The second author is partially supported by the NSF grant DMS-1400876 and the RFBR grant 15-01-99563.}

\address{School of Mathematical Sciences, Fudan University, Shanghai, 200433, P.R. China. }
 \email{zlu@fudan.edu.cn}
\address{University of Texas Rio Grande Valley, School of Mathematical and
 Statistical Sciences, One West University Boulevard, Brownsville, TX, 78520, USA.}
 \email{oleg.musin@utrgv.edu}

\begin{abstract}
In this paper we state some problems on rigidity  of powers in terms of complex analysis and number-theoretic abstraction, which has a strong topological background for the rigid Hirzebruch genera and Kosniowski's conjecture of unitary circle actions. However, our statements of these problems are elementary enough and do not require any knowledge of algebraic topology.  We shall give the solutions of these problems for some particular cases.
As a consequence, we obtain that Kosniowski's conjecture holds in the case of dimension $\leq 10$ or equal to 14.
\end{abstract}
\maketitle


\section{The power--rigidity  problem }

In this section we first give some simple notions  and propose some problems in terms of complex analysis and number-theoretic abstraction,  whose original ideas come from the topological background for the rigid Hirzebruch genera and Kosniowski's conjecture of unitary circle actions, as we will see in Section 3. This can lead us to work on the framework of complex analysis and number-theoretic abstraction, so that we can temporarily forget any  topological background. We hope that we can use this way to characterise all unitary closed ${\bf S^1}$-manifolds fixing only isolated points.

\subsection{$T_{x,y}$--rigidity.}
Let $W=\left(w_{ij}\right)$ be an $m\times n$ matrix, where all $w_{ij}$ are nonzero integers. Let $x$ and $y$ be any complex numbers. Define a complex function in variable $z$ as follows:

$$
T_{x,y}^W(z):=\sum\limits_{i=1}^m\prod_{j} {\frac{xz^{w_{ij}}+y}{z^{w_{ij}}-1}}.
$$
Then $T_{x,y}^W(z)$ is a rational function in $z$. We call numbers $w_{ij}$ {\em powers} or {\em weights}.

\begin{definition} \label{rigid}
We say that a set $W$ of powers  is {\bf rigid} (or $T_{x,y}$--{\bf rigid}) if  the function $T_{x,y}^W(z)$ is a constant.
\end{definition}

It is easy to prove the following statement.
\begin{proposition} \label{prop1} Let $W$ is rigid . Then
$$
T_{x,y}^W(z)\equiv \sum\limits_{i=1}^m x^{s_i^+}(-y)^{s_i^-}
$$
where $s_i^+, s_i^-$ are numbers of positive and negative powers $w_{ij}$ in the row  $w_{i1},\ldots, w_{in}$.
\end{proposition}

\begin{definition} Let $a_1, \ldots, a_{n+1}$ be distinct integer numbers. Let
$$(w_{i1},\dots,w_{in}):=(a_i-a_1,\ldots, a_i-a_{i-1},  a_i-a_{i+1},\ldots, a_i-a_{n+1}).$$
Then $W$ is an $(n+1)\times n$ matrix.   We say that this set of powers is {\bf quasilinear}.
\end{definition}

A proof of the following proposition can be found in  \cite{mus11}.
\begin{proposition} \label{qsl} Let $W$ is  quasilinear. Then  $W$ is rigid and
$$
T_{x,y}^W(z)=\frac{x^{n+1}-(-y)^{n+1}}{x+y}.
$$
\end{proposition}

 Topologically, the above case actually corresponds to the case of almost complex closed ${\bf S^1}$-manifolds fixing isolated points. Generally, it  can be extended to the case for sets of powers with signs, which has a more direct connection with $T_{x,y}$--genera of unitary circle actions fixing isolated points.  Let $W=\left(w_{ij}\right)$ be an $m\times(n+1)$ matrix, where all $w_{ij}$ are nonzero integers and $w_{i,n+1}=\pm1$. Write  $\varepsilon_i=w_{i,n+1}$. Define
$$
T_{x,y}^{W}(z):=\sum\limits_{i=1}^m\varepsilon_i\prod_{j} {\frac{xz^{w_{ij}}+y}{z^{w_{ij}}-1}}
$$
 where $W$ is called {\em a set of powers with signs}.
 \vskip .2cm
In a similar way to Definition~\ref{rigid}, we say that a set $W$ of powers with signs is {\em rigid} if $T_{x,y}^{W}(z)$ does not depend on $z$. In this case we also have
\begin{equation}\label{rig-e1}
T_{x,y}^{W}(z)=\sum\limits_{i=1}^m\varepsilon_i\prod_{j} {\frac{xz^{w_{ij}}+y}{z^{w_{ij}}-1}}\equiv
\sum\limits_{i=1}^m {\varepsilon_ix^{s_i^+}(-y)^{s_i^-}}.
\end{equation}
As we shall see,  this formula  is exactly  (\ref{e5}), induced by the rigidity of $T_{x,y}$--genus.


\begin{problem}[The power--rigidity problem] \label{prob}
Let $W$ be a rigid set of powers with signs or without signs (i.e., an
$m\times (n+1)$-matrix or an $m\times n$-matrix).
 \begin{enumerate}
 \item[$(i)$] What can be the set $W$ of powers?

\item[$(ii)$]  Suppose there are $x$ and $y$ such that $T_{x,y}^W(z)\ne0$.
 Is there some linear function $\ell(n)$ such that $m>\ell(n)$?  In particular,  is it true that $$m\ge [{n/2}]+1?$$
  \end{enumerate}
\end{problem}

 Problem~\ref{prob} (ii) is essentially related to  the Kosniowski's conjecture:
  for a unitary circle action on a unitary closed manifold $M$ of dimension $2n$ fixing just $m$ isolated points,
if $M$ is not a boundary, then  $m\ge [{n/2}]+1$.


\subsection{$L$--rigidity.}  Consider the case $x=y=1$. Let
$$
L_W(z):=T_{x,y}^W(z).
$$
Then
$$
L_W(z)=\sum\limits_{i=1}^m\varepsilon_i\prod_{j=1}^n {\frac{z^{w_{ij}}+1}{z^{w_{ij}}-1}}
$$

Note that, if we change a sign of any power  $w_{ij}$, then this is equivalent to  changing the sign of  $\varepsilon_i$. It is clear that without loss of generality we may assume that all powers $w_{ij}$, $j\le n,$ are {positive.} So here we may assume that $w=(w_{ij})$ is an $m\times n$ matrix with all $w_{ij}>0$. Let $s:=\{\varepsilon_1,\ldots,\varepsilon_m\}$ be  the set of signs. By $W=\{w,s\}$ we denote the set of weights and signs.

\begin{definition} We say that $W$ is {\bf $L$--rigid} if $L_W(z)$ does not depend on $z$.
\end{definition}

By definition we have that $T_{x,y}$--rigidity implies $L$--rigidity. However, the $L$--rigidity property is still  stronger  enough.

\begin{proposition}\label{prop3}
Let $W$ be $L$--rigid. Then  $L_W(z)=0$ and $m$ is even if $n$ is odd, and  $L_W(z)\equiv m \pmod{2}$ if $n$ is even.
\end{proposition}
\begin{proof} Note that $$L_W(z^{-1})=(-1)^nL_W(z).$$  Since  $L_W(z^{-1})=L_W(z)$ for all $L$--rigid $W$,   for odd $n$ we have $L_W(z)=0$.

\vskip .1cm

If $W$ with positive powers is L--rigid, then  (\ref{rig-e1}) yields that
\begin{equation*}\label{rig-e2}
L_W(z)=\varepsilon_1+\ldots+\varepsilon_m.
\end{equation*}
This implies the equality $L_W(z)\equiv m \pmod{2}$. In particular, since we have proved that if $n$ is odd, $L_W(z)=0$,  so
$m$ is even.
\end{proof}

\begin{corollary}
Let $W$ be $L$--rigid. If $L_W(z)\ne0$, then $n$ is even.
\end{corollary}

Note that if $W$ is quasilinear and $n$ is even, then it  follows from Proposition \ref{qsl}  that  $L_W(z)=1$ for even $n$. In this case $m=n+1$. Actually, we know only this example of $L$--rigid $W$ with $L_W(z)>0$ and $m\le n+1$.

\begin{problem}\label{prob1} Let $W$ be $L$--rigid with $m\le n+1$. Is it true that if  $L_W(z)\ne0$, then  $m=n+1$  and  $|L_W(z)|=1$? In particular, is $W$  quasilinear?
\end{problem}

\subsection{Pairs of powers.}
In \cite{mus80}, using $L$--rigidity and some simple topology, we proved that the set of weights $w=(w_{ij})$ of a circle action on a manifold can be divided into pairs $(w_{ij},w_{k\ell})$ such that $w_{ij}=w_{k\ell}$ with $i\ne k$ for some particular cases (see \cite[Theorem 1.1]{mus80}). Indeed, using only $L$--rigidity this can be carried out only for some particular cases.
\begin{problem} Let $W$ be an $L$--rigid. Is it true that the set of powers $w=(w_{ij})$  can be divided into pairs $(w_{ij},w_{k\ell})$ such that $w_{ij}=w_{k\ell}$ with $i\ne k$?
\end{problem}

\section{The power--rigidity  problem for some cases}

In this section we characterise the rigid sets of powers with signs in some special cases.

\subsection{The case $n=1$.}

\begin{theorem}\label{prop4} Let $W$ be $L$--rigid  with $n=1$.
Then $m=2k$ is even,  and  the sets of powers and signs are $$w=\{a_1,a_1,\ldots,a_k,a_k\}, \qquad s=\{1,-1,\ldots,1,-1\}. $$
\end{theorem}
\begin{proof} Proposition \ref{prop3} yields that $m$ is even and
$$
L_W(z)=\sum\limits_{i=1}^m\varepsilon_i{\frac{z^{b_i}+1}{z^{b_i}-1}} = 0.
$$

We may assume that $b_m\ge b_{m-1}\ge\ldots\ge b_1>0.$ Since $(z^{b_m}-1)^{-1}$ has poles at the $b_m$ roots of unity, there is an integer $l<m$ such that $b_l=b_m$ and $\varepsilon_l=-\varepsilon_m$. We can remove these two powers from $W$ and apply the same arguments for $W'$ with $m-2$ weights.
\end{proof}

\subsection{The case $m=2$.}
Consider the following three sets $W=\{(w_{ij})_{2\times n}, s=\{\varepsilon_1,\varepsilon_2\}\}$  of powers with signs:
\begin{enumerate}
 \item[$(Z):$] $w_{1i}=w_{2i}=a_i$, where $a_i\in{\Bbb Z}, \, a_i\ne0,$ for all $i=1,\ldots,n$,
and  $\varepsilon_1=-\varepsilon_2$.

 \item[$(L_1):$] $n=1$, $w_{11}=a, \, w_{21}=-a$, where $a>0$, and $\varepsilon_1=\varepsilon_2$.

 \item[$(S_3):$] $n=3$,   $(w_{11}, w_{12}, w_{13})=(a,b,-(a+b))=-(w_{21}, w_{22}, w_{23})$, where $a$ and $b$ are positive integer, and $\varepsilon_1=\varepsilon_2$.
\end{enumerate}

The following theorem has been proved in \cite{mus16}. Note that its proof is included here for a local completeness.

\begin{theorem}\label{weights}
Let $W$ be rigid  with $m=2$.  Then $W$ can only be $Z$, $L_1$ or $S_3$.
\end{theorem}

\begin{proof} Without loss of generality we may assume that
$$
|w_{k1}|\ge|w_{k2}|\ge \ldots\ge |w_{kn}|, \quad k=1,2.
$$
  It is easy to see that $|w_{1i}|=|w_{2i}|$
for all $i=1,\ldots,n.$
Moreover, if $W$ is not $Z$, then $w_{1i}=-w_{2i}$ for all $i$. Otherwise,
$T_{x,y}^W(z)$ cannot be a constant. Therefore, $T_{x,y}^W(z)$ can be written in the following form:

$$
T_{x,y}^W(z)=\frac{(xz^{a_1}+y)\cdots(xz^{a_k}+y)(x+yz^{b_1})\cdots(x+yz^{b_\ell})}
{(z^{a_1}-1)\cdots(z^{a_k}-1)(z^{b_1}-1)\cdots(z^{b_\ell}-1)} -
$$
$$
\frac{(x+yz^{a_1})\cdots(x+yz^{a_k})(xz^{b_1}+y)\cdots(xz^{b_\ell}+y)}
{(z^{a_1}-1)\cdots(z^{a_k}-1)(z^{b_1}-1)\cdots(z^{b_\ell}-1)}=x^ky^\ell-x^\ell y^k,
\eqno (2.1)
$$
where $k+\ell=n$ and all $a_i$ and $b_j$ are positive integers.

\vskip .1cm
So for $n=1$ we have that $W$ is $L_1$. If $n>1$, then $k>0$ and $\ell>0$. For $y=0$, the equation (2.1) implies
$$
a_1+\ldots+a_k=b_1+\ldots+b_\ell \eqno (2.2)
$$

Without loss of generality we may assume that $a:=a_1\ge a_i$ for all $i$ and $a_1>b_j$
for all $j$. (As we mentioned above, we cannot have the equality $a_1=b_j$.)
Let $x=-z^a$ and $y=1$. Then (2.1) yields
$$
\frac{(z^{2a}-1)\cdots(z^{a_k+a}-1)(z^{a}-z^{b_1})\cdots(z^{a}-z^{b_\ell})}
{(z^{a}-1)\cdots(z^{a_k}-1)(z^{b_1}-1)\cdots(z^{b_\ell}-1)}
=(-1)^\ell z^{ak}-(-1)^k z^{a\ell} \eqno (2.3)
$$
Since $a>b_j$, we have $a\ell>b_1+\ldots+b_\ell$. Then (2.3) implies
$$
ka=b_1+\ldots+b_\ell \eqno (2.4)
$$
Therefore, from (2.2) we have $a_1+\ldots+a_k=ka$, $a_1=\ldots=a_k=a$, $k$ is odd, and
$\ell$ is even. Then
$$
\frac{(z^{a}+1)^k(z^{a-b_1}-1)\cdots(z^{a-b_\ell}-1)}
{(z^{b_1}-1)\cdots(z^{b_\ell}-1)}
= z^{a(\ell-k)}+1 \eqno (2.5)
$$
Equation (2.5) yields $k=1$, $\ell=2$ and $a=b_1+b_2$. Thus, it is the case $S_3$.
\end{proof}

\subsection{The case $m=3$.}

\begin{theorem} \label{th23}
Let $W$ be $L$--rigid with $m=3$ and $n=2$.  Then $W$ is quasilinear.
\end{theorem}
\begin{proof}
In this case we have the following equation for six positive  integers $a_1$, $a_2$, $b_1$, $b_2$, $c_1$, $c_2$:
$$
\frac{(z^{a_1}+1)(z^{a_2}+1)}{(z^{a_1}-1)(z^{a_2}-1)}+ \frac{(z^{b_1}+1)(z^{b_2}+1)}{(z^{b_1}-1)(z^{b_2}-1)}=\frac{(z^{c_1}+1)(z^{c_2}+1)}{(z^{c_1}-1)(z^{c_2}-1)} +1 \eqno (2.6)
$$
Without loss of generality we can assume that $a_1\le a_2$, $b_1\le b_2$, $c_1\le c_2$ and $a_2\le b_2$. Let us show that $b_2>c_2$. Indeed, otherwise $c_2>b_2$ or $c_2=b_2$. If $c_2>b_2$, then (2.6) has poles at the $c_2$ roots of unity. If $c_2=b_2$, then $c_2\ge a_2$ and
$$
(u^{a_1}+1)(u^{a_2}+1)=(u^{a_1}-1)(u^{a_2}-1), \quad \mbox{ i.e. } \quad u^{a_2-a_1}+1=0,
$$
where $u$ is any of the roots of the equation $z^{c_2}+1=0$, a contradiction.

We have that $b_2$ is the maximum power.  Since (2.6) has no poles at the $b_2$ roots of unity, there is one more maximum  power  and it is $a_2$, i.e. $a_2=b_2$. Therefore, equation (2.6) can be written in the following form
$$
\frac{(z^{a_2}+1)(z^{a_1+b_1}-1)}{(z^{a_2}-1)(z^{a_1}-1)(z^{b_1}-1)}=\frac{z^{c_1+c_2}+1}{(z^{c_1}-1)(z^{c_2}-1)}  \eqno (2.7)
$$
Equation (2.7) implies that $a_1+b_1=c_1+c_2=a_2$ and $a_1=c_1$, $b_1=c_2$ or $a_1=c_2$, $b_1=c_1$. Thus, $W$ is quasilinear.
\end{proof}

It is not difficult to see that for the case $n=2$ and  any $m$,  by similar arguments as above, all $W$ can be classified. However, the case of $m=3$ and any even $n$ looks more difficult  for which we do not have a ready solution now. Actually, this problem is equivalent to the following:
\begin{problem}  Consider an equation for $3n$ positive  integers $a_i$,  $b_i$,  and $c_i$, where $i=1,\ldots, n$:
$$
\frac{(z^{a_1}+1)\ldots(z^{a_n}+1)}{(z^{a_1}-1)\ldots(z^{a_n}-1)}+ \frac{(z^{b_1}+1)\ldots(z^{b_n}+1)}{(z^{b_1}-1)\ldots(z^{b_n}-1)}=\frac{(z^{c_1}+1)\ldots(z^{c_n}+1)}{(z^{c_1}-1)\ldots(z^{c_n}-1)} +1.
$$
Find the largest $n$ for which there exists a solution of this equation. Is it true that there are no solutions for $n>2$?
\end{problem}


\section{Unitary circle action, Kosniowski conjecture and  rigid Hirzebruch genera}

\subsection{Unitary circle action and Kosniowski conjecture}
We say that $M$ is a  {\it unitary ${\bf S}^1$-manifold} if it is a smooth closed manifold
with an effective circle action such that its tangent bundle admits an
${\bf S}^1$-equivariant stable complex structure.

\vskip .2cm
Each component of the fixed point set of a unitary ${\bf S}^1$-manifold $M$ is again a unitary manifold, and its normal bundle to $M$ is a complex ${\bf S}^1$-vector bundle with a complex structure induced from the one on $TM\oplus\underline{\mathbb{R}}^r$. In particular, the tangent space $T_pM$ at an isolated ${\bf S}^1$-fixed point $p$ is a complex
${\bf S}^1$-module. Thus, if $M$ has an isolated ${\bf S}^1$-fixed point, then the dimension of $M$ must be even.

\vskip .2cm

Let $M$ be a unitary $2n$-dimensional ${\bf S}^1$-manifold $M$ with $m$ isolated fixed
points $p_1,\ldots,p_m$.
Denote  weights in the tangent ${\bf S}^1$-representation at $p_i$ by $w_{i1},\ldots,w_{in}$ and its sign
by $\varepsilon_i$, where $\varepsilon_i=\pm 1$. For the notions of weights and signs at isolated points, see~\cite{Kos, BPR, BP}.

\vskip .2cm

Bott residue formula  tells us how to calculate the Chern numbers of $M$ in terms of the set of weights and signs:

\begin{equation}\label{bott}
\langle c_1^{r_1}\cdots c_n^{r_n}, [M]\rangle
=\sum_{i=1}^m{{\sigma_1^{r_1}(w_{i1},\ldots,w_{in})\cdots \sigma_n^{r_n}(w_{i1},\ldots,w_{in})}\over{\varepsilon_i\prod_{j=1}^nw_{ij}}}
\end{equation}
where $r_1+2r_2+\cdots+nr_n=n$ and $\sigma_i$ denotes the $i$-th elementary symmetric function.
Note that if $r_1+2r_2+\cdots+nr_n<n$, then $\langle c_1^{r_1}\cdots c_n^{r_n}, [M]\rangle=0$. This implies that
we obtain many equations for the set of weights and signs (cf. \cite{AS}, \cite{B}).

 \vskip .2cm

 Kosniowski in \cite{Kos} posed the following conjecture.
  \begin{conj}[Kosniowski]\label{conj}
   A unitary $2n$--dimensional ${\bf S}^1$--manifold
 that is not a boundary has at least $[{n/2}]+1$ isolated fixed points.
\end{conj}

In his seminar paper~\cite{Kos}, Kosniowski showed that the conjecture holds if the number of isolated points is less than three.

\vskip .2cm

 Recently some related works have been carried on with respect to this conjecture.
For example,  Pelayo and Tolman in~\cite{pt} considered compact
symplectic manifolds $M^{2n}$ with symplectic circle actions fixing isolated points, and
showed that if the weights  satisfy some subtle condition, then the action
has at least~$n+1$ isolated fixed points. In~\cite{ll}, Ping Li and Kefeng
Liu showed that if~$M^{2mn}$ is an almost complex manifold such that there is some nonzero Chern number~$\langle(c_{\lambda_1}\dots
c_{\lambda_r})^n, [M]\rangle$ where $\lambda=(\lambda_1,\dots,\lambda_r)$ is a partition of
$m$, then for any~$S^1$-action on~$M$,
it must have at least~$n+1$ fixed points.

\vskip .2cm

 On the other hand, for unitary
 ${\bf T}^n$--manifolds $M$ of dimension $2n$, it was showed in \cite{LuTan}  that if $M$ is non-bounding, then
 there are at least $\lceil{n/2}\rceil+1$ isolated fixed points (see also \cite{Lu13,Lu}).
 It is known from \cite[Lemma 4.2.1]{AP} or \cite[Lemma 7.4.3]{BP} that if $M$ is a connected smooth orientable closed manifold admitting an action of torus group $T^k$, then there is a circle subgroup $S<T^k$ such that $M^S=M^{T^k}$,  which is essentially based upon \cite[Theorem 10.5, Chapter IV]{Bre}.
 This means that Kosniowski conjecture holds in the setting of unitary
 ${\bf T}^n$--manifolds $M$ of dimension $2n$,  providing a support evidence.

\vskip .2cm

The following theorem also gives an affirmative answer to Kosniowski conjecture in the case of $n\leq 5$ or $n=7$.

\begin{theorem} \label{K-conj}
If $n\leq 5$ or $n=7$, then a unitary $2n$-dimensional ${\bf S}^1$-manifold
 that is not a boundary must have at least $[{n/2}]+1$ isolated fixed points.
\end{theorem}

 Our proof relies on the fact that $T_{x,y}$--genus is rigid. We shall introduce this method below.
However, it is not clear if Kosniowski's conjecture can be proved using only the rigidity of the
$T_{x,y}$--genus.

\begin{remark}
It should be pointed out that in the setting of almost complex closed ${\bf S^1}$-manifolds  fixing only isolated points, Jang showed in \cite{J} that if an almost complex closed ${\bf S^1}$-manifold $M$ fixes only three isolated points, then $\dim M$ must be 4. This implies that Kosniowski conjecture holds  in the setting of almost complex closed ${\bf S^1}$-manifolds of dimension $\leq 14$ fixing only isolated points.
\end{remark}


\subsection{Rigid Hirzebruch genera}

Let $U_*$ be the complex bordism ring with coefficients in $R={\Bbb Q}$, ${\Bbb R}$, or
${\Bbb C}$. For a closed smooth stably complex manifold $M$, Hirzebruch \cite{Hir1}
defined a multiplicative genus $h(M)$ by a homomorphism $h:U_*\otimes R\to R$.

\vskip .2cm

Recall that according to Milnor and Novikov, two stably complex  manifolds are complex
cobordant if and only if they have the same Chern numbers (see \cite{M,N}). Therefore, for any
multiplicative genus $h$ there exists a multiplicative sequence of polynomials
$\{K_i(c_1,\ldots,c_i)\}$  such that
$
h(M)=K_n(c_1,\ldots,c_n),
$
where the $c_k$ are the Chern classes of $M$ and $n=\dim_{\Bbb C}(M)$ (see \cite{Hir1}).

\vskip .2cm
Let $U_*^G$ be the ring of complex bordisms of manifolds with actions of a compact Lie
group $G$. Then for any homomorphism $h:U_*\otimes R\to R$ can be define  an
{\it equivariant genus} $h^G$, i.e. a homomorphism
$$
h^G:U_*^G\otimes R\to K(BG)\otimes R
$$
(see details in \cite{krich1}).

\vskip .2cm

A multiplicative genus $h$ is called {\it rigid} if for any connected compact group $G$
the equivariant genus $h^G(M)=h(M)$. That means
$$
h^G:U_*^G\otimes R\to R\subset K(BG)\otimes R,
$$
i.e. $h^G([M,G])$ belongs to the ring of constants. It is well known
(see \cite{AH, krich1}) that ${\bf S}^1$-rigidity implies $G$-rigidity, i.e. it is
sufficient to prove rigidity only for the case $G={\bf S}^1$.

\vskip .2cm

For $G={\bf S}^1$, the universal classifying space $BG$ is ${\Bbb C}${P}$^\infty$, and
the ring $K(BG)\otimes R$ is isomorphic to the ring of formal power series  $R[[u]]$. Then
for any ${\bf S}^1$-manifold $M$ and a Hirzebruch genus $h$ we have
$h^{S^1}([M,{\bf S}^1])$ in $R[[u]]$. In the case of an ${\bf S}^1$-action on a unitary
manifold $M^{2n}$ with isolated fixed points $p_1\ldots,p_m$ with weights $
w_{i1},\ldots,w_{in}$, and signs $\varepsilon_i$, $i=1,\ldots,m$,  $\; h^{S^1}$ can be
found explicitly (see \cite{AH,BP,BPR,BR,krich1}):
\begin{equation}\label{e1}
h^{S^1}([M,{\bf S}^1])=S_h(\{w_{ij}\},u):=\sum\limits_{i=1}^m\varepsilon_i\prod_{j}
{\frac{H(w_{ij}\,u)}{w_{ij}u}},
\end{equation}
where $H$ is the characteristic series of $h$ (see \cite{Hir1,mus11}). If $h$ is rigid, then
from $(\ref{e1})$ it follows that
\begin{equation}\label{e2}
h(M)=S_h(\{w_{ij}\},u) \; \mbox{ for any } \; u.
\end{equation}

Atiyah and  Hirzebruch  \cite{AH} based on the Atiyah-Singer index theorem  proved that $L$--genus (signature) is rigid for oriented $S^1$--manifolds and
$T_y-$genus is rigid for almost complex $S^1$--manifolds.  Krichever \cite{krich1} gave a proof of rigidity of the
$T_{x,y}-$genus for almost complex $S^1$--manifolds using global analytic properties of $S_h(\{w_{ij}\},u)$. In \cite{BP,BPR,BR} this result was extended for unitary $S^1$--manifolds. Note that in \cite{mus11} it was showed that if $h$ is a rigid genus, then it is $T_{x,y}-$genus.

\vskip .2cm

Now it is not hard to see that $(3.2)$ yields the
Atiyah-Hirzebruch formula for a unitary ${\bf S}^1$-manifold $M$:
\begin{equation}\label{e3}
T_{x,y}(M)=\sum\limits_{i=1}^m {\varepsilon_ix^{s_i^+}(-y)^{s_i^-}}
\end{equation}
where $s_i^+, s_i^-$ are numbers of positive and negative weights $\{w_{ij}\}$.
In \cite{krich1} Krichever gave the following formula for the $T_{x,y}$--genus:
 \begin{equation}\label{e4}
 H_{x,y}(u)=\frac{u(xe^{u(x+y)}+y)}{e^{u(x+y)}-1}.
 \end{equation}
Let $z:=e^{(x+y)u}$. Then (\ref{e1}), (\ref{e2}),  (\ref{e3}) and (\ref{e4}) imply
\begin{equation}\label{e5}
\sum\limits_{i=1}^m\varepsilon_i\prod_{j} {\frac{xz^{w_{ij}}+y}{z^{w_{ij}}-1}}\equiv
\sum\limits_{i=1}^m {\varepsilon_ix^{s_i^+}(-y)^{s_i^-}}.
\end{equation}

\begin{remark}
It should be pointed out that the formula (\ref{e5}) with signs was first proved by Buchstaber--Panov--Ray in the 2010 paper~\cite{BPR}, and a shorter and more conceptual proof was given in \cite[Theorem 9.4.8]{BP}.
\end{remark}


\subsection{Rigidity of powers and $S^1$-actions on manifolds.}
Equation \ref{e5} provides the topological background of Section 1. Namely, the set of weights $W$ of  an oriented $S^1$-manifold is $L$--rigid and for unitary actions $W$ is $T_{x,y}$--rigid.

Let $M^{2n}$ be a compact oriented $S^1$-manifold with $m$ isolated fixed points. Then for the set of weights $W$ we have
$$
L_W(z)=\sum\limits_{i=1}^m{\varepsilon_i}=L(M).
$$
It is easy to see that Theorem \ref{th23} yields the following theorem.

\begin{theorem} Let $M$ be a compact  4-dimensional oriented $S^1$--manifold with three isolated fixed points. Then this $S^1$-action must be quasilinear.

\end{theorem}

Suppose that $M^{2n}$ is a non-bounding unitary $2n$--dimensional ${\bf S}^1$-manifold fixing only $m$ isolated points $p_1, ..., p_m$ with weights $
w_{i1},\ldots,w_{in}$, and signs $\varepsilon_i$, $i=1,\ldots,m$. Then the $T_{x,y}$--rigidity implies the following equation:
$$
T_{x,y}^W(z)=T_{x,y}(M).
$$

For the case $m=2$, Theorem~\ref{weights} (see also \cite{mus16} or \cite[Theorem 5]{Kos}) tells us that  there are the following possibilities for the unitary ${\bf S}^1$-manifold $M^{2n}$:
\begin{enumerate}
 \item[$(Z):$]
 $n$ is a arbitrary integer,  $w_{1i}=w_{2i}=a_i$ where $a_i\in{\Bbb Z}, \, a_i\ne0,$ for all $i=1,\ldots,n$,
and  $\varepsilon_1=-\varepsilon_2$.
 \item[$(L_1):$] $n=1$, $w_{11}=a, \, w_{21}=-a$, where $a>0$, and $\varepsilon_1=\varepsilon_2$.
 \item[$(S_3):$] $n=3$, weights in $p_1$ and $p_2$ are $(a,b,-(a+b))$ and $(-a,-b,a+b)$
respectively, where $a$ and $b$ are positive integer, and $\varepsilon_1=\varepsilon_2$.

\end{enumerate}


\subsection{Proof of Theorem~\ref{K-conj}}
By Bott residue formula (\ref{bott}), we have that any unitary ${\bf S}^1$-manifold $M^{2n}$ with weights and signs satisfying $(Z)$ must be a boundary.
We can also prove a stronger fact that any unitary ${\bf S}^1$-manifold $M^{2n}$ with weights and signs satisfying  $(Z)$ bounds equivariantly, i.e., it is cobordant to zero in the ring $U_*^{{\bf S}^1}$ of
complex bordisms  with circle actions. In fact, using Atiyah-Bott-Berline-Vergne localization formula, we can easily show that
all equivariant Chern numbers of any unitary ${\bf S}^1$-manifold $M^{2n}$ with weights and signs satisfying $(Z)$ vanish, and it then follows from a theorem of Guillemin-Ginzburg-Karshon \cite{GGK} that $M^{2n}$ bounds equivariantly.

\vskip .2cm
 In \cite{Kos} Kosniowski also gave the examples of $M_1^2$ satisfying $(L_1)$ and $M_2^6$ satisfying $(S_3)$.
 Actually, $M_1^2$ satisfying $(L_1)$ can be  just  ${\Bbb C}${P}$^1$ with a linear action of ${\bf S}^1$:
 $$
 [z_0:z_1] \longmapsto  [e^{ia\varphi}z_0:z_1]
 $$
and $M_2^6$ satisfying $(S_3)$ can be the 6-sphere $S^6$ with a circle action. Now by Bott residue formula (\ref{bott}),
a direct calculation shows that $\langle c_1, [M_1^2]\rangle=2$ and $\langle c_3, [M_2^6]\rangle=2$, so both $M_1^2$ and $M_2^6$ are not  boundaries.

\vskip .2cm
It should also be pointed out that any unitary ${\bf S}^1$-manifold $M^{2n}$ fixing exactly an isolated point is impossible.
In fact, by Bott residue formula (\ref{bott}), we obtain a contradictive equation
$$0=\langle c_0, [M^{2n}]\rangle=\varepsilon_1{{1}\over{\prod_{j=1}^nw_{1j}}}\not=0.$$

\vskip .2cm
Together with the above arguments, we  see that for a unitary ${\bf S}^1$-manifold $M^{2n}$ fixing just two isolated points,
if $M^{2n}$ is not a boundary, then $n=1$ or $3$, satisfying that $[n/2]+1\leq 2$. This also implies that  for a unitary ${\bf S}^1$-manifold $M^{2n}$ fixing $m$ isolated points with $n\not=1,3$, if $M^{2n}$ is not a boundary, then $m$ must be at least 3, and in particular, when $n\not=1,3$ is odd, $m$ is an even number greater than or equal to 4 by Proposition~\ref{prop3}.
On the other hand, the  solution of inequality $[n/2]+1\leq 3$ is $n\leq 5$, and $[7/2]+1=4$.  This completes the proof of Theorem~\ref{K-conj}.\qed


\end{document}